\documentclass{article}
\usepackage[utf8]{inputenc}
\usepackage{amsmath}
\usepackage{amssymb}
\usepackage{amsthm}
\usepackage{amsfonts}
\usepackage{authblk}
\usepackage{hyperref}

\title{Completion Problems and  Sparsity  for Kemeny's Constant}
\date{\today}

\author[1]{Stephen Kirkland}
\affil[1]{Department of Mathematics, University of Manitoba, Winnipeg, MB, Canada, \href{}{stephen.kirkland@umanitoba.ca}.}

\begin{document}
\def\ones{{\bf{1}}}
\newcommand{\K}{{\cal{K}}} 
\newcommand{\reals}{\mathbb{R}} 
\newcommand{\trace}{{\rm{trace}}} 
\newcommand{\diag}{{\rm{diag}}} 
\def\S{{\cal{S}}} 
\newcommand{\naturals}{\mathbb{N}} 
\def\t{{\mathcal{t}}}
\newtheorem{theorem}{Theorem}[section]
\newtheorem{definition}[theorem]{Definition}
\newtheorem{proposition}[theorem]{Proposition}
\newtheorem{corollary}[theorem]{Corollary}
\newtheorem{observation}[theorem]{Observation}
\newtheorem{lemma}[theorem]{Lemma}
\newtheorem{example}[theorem]{Example}
\newtheorem{problem}[theorem]{Problem}
\newtheorem{remark}[theorem]{Remark}
\maketitle
\begin{abstract}
For a partially specified stochastic matrix, we consider the problem of completing it so as to minimize Kemeny's constant. We prove that for any   partially specified stochastic matrix  for which the problem is well-defined,  there is a minimizing completion that is as sparse as possible. We also find the minimum value of Kemeny's constant in two special cases: when the diagonal has been specified, and when all specified entries lie in a common row.
\end{abstract}


\section{Introduction and motivation}\label{sec:intro}

%need a killer first sentence 

We live in a world of incomplete information. Consequently, we face the challenge: based only on partial information,  can we achieve some desirable objective, and if so how?

That question is at the core of many matrix completion problems. Recall that a \emph{partial matrix} is one in which only some of the entries are specified, while a matrix that arises by assigning values to each of the unspecified entries is known as a \emph{completion} of the partial matrix. 
There is a 
wealth of literature on matrix completion problems, often focusing on   completions that possess some desirable property (e.g. positive definiteness  or total nonnegativity) or that optimize some scalar function (e.g. rank or determinant). See \cite{hla, Jo} for example, for overviews of the literature on matrix completion problems. 

%cite some surveys eg charlie, leslie, charlie/shaun book. look al%so in HLA for references. 

In this paper we adopt the perspective of completion problems to the topic of stochastic matrices. Recall that a square matrix is \emph{stochastic} if it is entrywise nonnegative, and each row sum is equal to $1$. Any stochastic matrix has $1$ eigenvalue, with $\ones,$ the all-ones vector of the appropriate order, as a corresponding eigenvector. Stochastic matrices are central to the study of Markov chains,  a rich and widely applied class of stochastic processes. Each partial matrix $P$ that we will consider satisfies the following properties: \\
i) $P$ is square;
ii) all specified entries of $P$ are nonnegative;  
iii) in each row of $P$,  the sum of the specified entries is bounded above by $1,$ and is equal to $1$ for any fully specified row;
iv) any row of $P$ for which the sum of the specified entries is equal to $1$ is in fact fully specified;
v) any row of $P$ that is not fully specified contains at least two unspecified entries.  \\
We say that a partial matrix $P$ satisfying i)--v) is a \emph{partial stochastic matrix}. 
Evidently any partial stochastic matrix has a completion to a stochastic matrix. We note that conditions iv) and v) above are imposed largely to avoid uninteresting cases. If there is a row of $P$ for which the sum of the specified entries is equal to $1$, then any stochastic completion necessarily assigns the value $0$ to the unspecified entries in that row (so those positions of $P$ are specified de facto). Similarly, if the $j$-th row of $P$ has just one unspecified entry, say in the $k$-th position  then any stochastic completion  assigns the value $1-\sum_{\ell \ne k}p_{j,\ell}$ to the $(j,k)$ position (so that the entire $j$-th row of $P$ is specified de facto). 

%possibly say that the matrix is partially stochastic, or partially %substochastic??

Next, we briefly discuss Kemeny's constant (which first appeared in \cite{KS}), the scalar function that will be minimized in this paper. Given an irreducible $n\times n$ stochastic matrix $T$ and indices $j,k=1, \ldots, n,$ the entry $t_{j,k}$ is the probability that the Markov chain transitions from state $j$ to state $k$ in one time step. Recall that there 	is a unique positive left eigenvector $w^\top$ whose entries sum to $1$ such that $w^\top T=w^\top$; that vector is known as the \emph{stationary vector} for the corresponding Markov chain. For each pair of indices $j,k=1, \ldots, n,$ the \emph{mean first passage time from j to k}, $m_{j,k}$ is the expected number of steps for the Markov chain to enter state $k$ for the first time, given that it began in state $j$ (we note in passing that $m_{j,j}=\frac{1}{w_j}$ for each such $j$). \emph{Kemeny's constant}, denoted $\K(T),$ is then given by $\sum_{k \ne j} m_{j,k}w_k$. It turns out  that $\K(T)$ does not depend on the choice of $j$, and $\K(T)$ can be interpreted  in terms of the expected number of steps required for the Markov chain to reach a randomly chosen state, starting from state $j$. Since  $\sum_{j=1}^n w_j=1$, we find that $\K(T) = \sum_{k,j, k \ne j} w_jm_{j,k}w_k.$ Hence $\K(T)$ also reflects the expected number of steps required to transition from a randomly chosen initial state to a randomly chosen destination state, and so measures the overall ease of movement from state to state. 
For this reason, Kemeny's constant has been proposed as an indicator of the  average travel time in a Markov chain model for vehicle traffic networks \cite{CKS}. 

There are several different  formulas for Kemeny's constant, and each has proven itself to be useful depending on the context. Setting $Q=I-T$ we find that $Q$ is singular, and when $T$ is irreducible, $Q$  has $0$ as an eigenvalue of algebraic multiplicity one. Hence, while $Q$ is not invertible, it has a group inverse,  $Q^\#$, and it transpires that $\K(T)= \trace (Q^{\#})$  (see \cite{KN} for background on the group inverse and its use in the analysis of Markov chains).  From this we find readily that if $T$ has eigenvalues $1=\lambda_1, \lambda_2, \ldots, \lambda_n$ (including multiplicities), then $\K(T)= \sum_{j=2}^n \frac{1}{1-\lambda_j}$ (\cite{ll}), and that if $Q=XY$ is a full rank factorization of $Q$, then since $Q^\#=(X(YX)^{-2}Y,$ we have  $\K(T)=\trace (X(YX)^{-2}Y) = \trace ((YX)^{-1}).$ 
We note that while Kemeny's constant was originally defined for the case that $T$ is irreducible (so that all of the mean first passage times are well-defined), the definition can be extended to cover the case that $T$ has $1$ as an algebraically simple eigenvalue. In that setting we may use any of $\trace (Q^{\#}), \trace ((YX)^{-1}), $ and $
\sum_{j=2}^n \frac{1}{1-\lambda_j}$ as the definition for $\K(T).$ 

For any stochastic matrix $T$, 
the multiplicity of $1$ as an eigenvalue can be understood from a combinatorial viewpoint. 
Recall that a state $j$ of an $n$-state Markov chain is \emph{essential} if it has the property that, for any vertex $k$ in the directed graph of $T$,  if there is a walk from $j$ to $k$ then necessarily there is a walk back from $k$ to $j$. The essential states can thus be  partitioned into  \emph{essential classes}, i.e. subsets of states such that a) within each class any state can be reached from any other state in the class, and b) no state from  
outside the class can be reached from any state inside the class. (See \cite{sen}, for example, for further details.)  It is well-known that the algebraic multiplicity of $1$ as an eigenvalue of the corresponding transition matrix coincides with the number of essential classes. Hence, $1$ is an algebraically simple eigenvalue of a stochastic matrix $T$ if and only if  there is a single essential class of states in the Markov chain associated with $T$. 

In view of the interpretation of Kemeny's constant as a measure of the  ease of movement between states, there is interest in identifying  stochastic matrices for which the corresponding value of Kemeny's constant is small. In \cite{hunter} it is shown that for any irreducible  stochastic matrix of order $n$, Kemeny's constant is  bounded below by $\frac{n-1}{2}$; \cite{kfast} proves that the adjacency matrix of a directed $n$-cycle is the unique irreducible  stochastic matrix of order $n$ with  $\frac{n-1}{2}$ as the value for Kemeny's constant. In a related vein, 
 \cite{kfast} identifies the stochastic matrix $T$ whose directed graph is subordinate to a given directed graph and has the smallest value for $\K(T),$ while \cite{Kstat}  identifies the stochastic matrices with prescribed stationary distribution that minimize Kemeny's constant. 

In each of \cite{hunter}, \cite{kfast} and \cite{Kstat}, the minimizing matrices are \emph{sparse}, i.e. the ratio of zero entries to nonzero entries is high. At least part of our agenda in the present work is to develop a better understanding of how the sparsity of a stochastic matrix affects the value of Kemeny's constant. The following observation goes in that direction. 

\begin{observation} {\rm {Suppose that we have an irreducible stochastic matrix $T$ of order $n$, with corresponding mean first passage matrix $M$ and stationary vector $w^\top$. In 
	\cite{kczech}, the following quantity, known as an accessibility index, is defined for each $k=1, \ldots, n$: $\alpha_k = w^\top M e_k -1$. Observe that $\K(T) = \sum_{k=1}^n w_k \alpha_k.$  
	
	For each $k=1, \ldots, n,$ let $R_k$ denote the random variable that is the time of the first return to state $k$. It is well-known that $ w_k=\frac{1}{E(R_k)},$ and in \cite{kczech} it is shown that 
	$\alpha_k= \frac{1}{2}\left( \frac{E(R_k^2)}{E(R_k)} -1\right), k=1, \ldots, n.$ Hence $
	w_k \alpha_k = \frac{1}{2}\left( \frac{E(R_k^2)}{(E(R_k))^2} -w_k\right)
	= \frac{1}{2}\left( \frac{E(R_k^2)-(E(R_k))^2 }{(E(R_k))^2}\right) +\frac{1-w_k}{2} = \frac{w_k^2}{2}Var(R_k)+\frac{1-w_k}{2},$ 
	and we 
	deduce that 
	\begin{equation}\label{var} 
			\K(T) =\frac{1}{2} \sum_{k=1}^n w_k^2Var(R_k) + \frac{n-1}{2}. 
	\end{equation}
The presence of $Var(R_k)$ in \eqref{var} provides some intuition as to how sparsity plays a role, as zeros in the transition matrix $T$ may serve to decrease the variances of the first return times. Indeed, in the extreme case that $\K(T)=\frac{n-1}{2},$  $T$ has just one nonzero entry in each row and $Var(R_k)=0, k=1, \ldots, n$. 	
}}
\end{observation}

%(n-1)/2 result, leading to the variance formula. 

%eg complete to psd, to low rank, maybe some others. mention a few surveys. 

%kemeny's constant literature, special interest in low values of kemeny's constant as these correspond to MCs with good circulation properties. The traffic model as an application, but also as a motivator for interesting questions. min kem with fixed stat vec, or with fixed digraph. 

In this paper we consider the following class of questions. Given a partial stochastic matrix $P$, over all stochastic completions of $P$, 
what is the minimum value of Kemeny's constant, and which completions yield that minimum? In section \ref{prelim} we show that, provided that $P$ admits a completion for which Kemeny's constant is well-defined, 
  there is a completion that minimizes Kemeny's constant and has the property that in each row, at most one of the unspecified entries is assigned a positive value. In particular, such a minimizing completion is as sparse as possible. In the remaining sections we solve two specific completion problems; section \ref{sec:diag} addresses the completion problem for the case that the partial stochastic matrix has  its specified entries on the diagonal, while section \ref{sec:row} covers the case that the partial stochastic matrix has  its specified entries in a common row. The solutions to both completion problems correspond to sparse transition matrices, as  anticipated by the results of section \ref{prelim}. Both completion problems help to explain how   local information (probability of transition back to the same state, and probabilities of transitions out of a single state, respectively) influences the global behaviour of the Markov chain as measured by Kemeny's constant.

We note that the techniques introduced in \cite{kfast} inform much of our approach. Indeed, some of the results in section \ref{prelim} are modest adaptations of the results in that paper. This is not so surprising, as the main problem in \cite{kfast} is to minimize Kemeny's constant over all stochastic matrices whose directed graph is a subgraph of a given directed graph $D$. This in turn can be rephrased as a completion problem: the relevant partial matrix has a prescribed entry $0$ in position $(j,k)$ for each pair of indices $j,k$ such that the arc $j \rightarrow k$ is not present in $D$. Evidently the completions of that partial matrix correspond to stochastic matrices with directed graphs subordinate to $D$.

%********** interp/motn for the two problems** from the traffic model perspective, diagonal is the prob of staying on the same road segment, specified row is complete information on transitions starting from a single location. 

Throughout we rely on a mix of combinatorial and matrix analytic techniques. We  also employ basic results and terminology  from nonnegative matrix theory and the study of Markov chains. The  reader is referred to \cite{sen} and \cite{KS} for the necessary background.

%%%%%%%% may need to clarify essential class of states/indices
%techniques from matrix analysis and combinatorics. 

\section{Preliminaries}\label{prelim}

%technical prelims, esp stuff from digraph paper. 

%"A partial matrix is a matrix in which some entries are specified and others %are
%not. A completion of a partial matrix is a specific choice of values for the %unspecified
%entries."

%sparsity, variance, given digraph\\

%completion problems, digraph as an example, patterns that are allowable \\

%%%notation P(A,A^c)

Our interest in this paper is Kemeny's constant, which is only well-defined for a stochastic matrix having a single essential class of indices. For this reason, we first identify the partial stochastic matrices for which at least one completion has a single essential class of indices. 
%The following result does precisely that. 
Here we use the notation that for a matrix $A$ and nonempty subsets $X, Y$ of its row and column indices respectively, $A[X,Y]$ denotes the submatrix of $A$ on rows indexed by $X$ and columns indexed by $Y$. 

\begin{proposition}
	\label{prop:admit} 
	Suppose that $P$ is a partial stochastic matrix of order $n$. There is a completion of $P$  having a single essential class of indices if and only if for each non-empty $X \subset \{1, \ldots, n\},$ one of $P[X, X^c]$ and $P[X^c,X]$ contains either a positive entry or an unspecified entry. 
\end{proposition}
\begin{proof}
	We first note that  a stochastic matrix $T$ has a single essential class if and only if for each non-empty proper subset $X$ of its index set, either $T[X,X^c]$ or $T[X^c,X]$ has a positive entry. 
	
	Let $P$ be a partial stochastic matrix and suppose that $T$ is a completion with a single essential class of indices. Fix a non-empty subset $X \subset \{1, \ldots, n\}$. Then one of $T[X,X^c]$ and $T[X^c,X]$ has a positive entry, from which we conclude that  one of $P[X, X^c]$ and $P[X^c,X]$ contains either a positive entry or an unspecified entry.  
	
	Conversely, suppose that for each non-empty $X \subset \{1, \ldots, n\},$ one of $P[X, X^c]$ and $P[X^c,X]$ contains either a positive entry or an unspecified entry. Let $T$ be a completion of $P$ having the largest number of positive entries; from properties i)--v) we find that $T$ assigns a positive number to each unspecified entry in $P$. It now follows readily that $T$  has a single essential class of indices. 
\end{proof}

For a partial stochastic matrix $P$, let $$SC(P)=\{T| T {\mbox{ is a stochastic  completion of }} P \},$$  and in the case that there is a stochastic completion of $P$ with a single essential class, we define $m(P)= \inf \{\K(T)| T \in SC(P) {\mbox{ and has a single essential class}}\}.$  In this section we establish some technical lemmas, and prove a general result (Proposition \ref{entries} below) regarding completions that realize $m(P).$ The following result is a slight adaptation of 
Lemma 2.3 in \cite{kfast}. We omit the proof as it is essentially found in \cite{kfast}. 

\begin{lemma}\label{eigbd} Let $T$ be a stochastic matrix of order $n$ having $1$ is an algebraically simple eigenvalue. If $\lambda \ne 1$ is an eigenvalue of $T$, then $|1-\lambda| \ge \frac{1-\cos(\frac{2\pi}{n})}{\K(T)}.$ 
\end{lemma}

 The proof of the next result is basically the same as that of  Lemma 2.4 in \cite{kfast} (but uses Lemma \ref{eigbd} above). 
 
\begin{lemma}
	\label{min} Let $P$ be a partial stochastic matrix for which there is a stochastic completion having a single essential class of indices. 
	There is a $T \in SC(P)$ having a single essential class such that $\K(T) = m(P).$ 
\end{lemma}

The following result is a slight extension of Lemma 2.5 in \cite{kfast}. We include a short proof in order to illustrate a key technique. 

\begin{proposition}\label{entries} 
	Let $P$ be a partial stochastic matrix for which there is a stochastic completion having a single essential class of indices.
	%, and let $\S$ denote the corresponding set of specified positions. 
	There is a $T \in SC(P)$ such that\\
	i) $T$ has a single essential class, \\
	ii)  $\K(T) = m(P),$  and \\
	iii) for each $j=1, \ldots, n$ there is at most one $k$ such that  $p_{j,k}$ is unspecified and $t_{j,k}>0.$ 
\end{proposition}
\begin{proof}
	From Lemma \ref{min}, there is a $T \in SC(P)$ satisfying i) and ii). If $T$ also satisfies iii), we're done, so suppose that for some indices $j, k_1, k_2$ the  $(j,k_1), (j,k_2) $ entries of $P$ are unspecified and $t_{j,k_1}, t_{j,k_2}>0.$  Our goal is to construct a matrix $\hat T$ satisfying i) and ii) such that $\hat T$ has fewer positive entries than $T$ does. The conclusion will then follow by an induction step on the number of positive entries. 
	
	For each $x \in [-t_{j,k_1}, t_{j,k_2}],$ let $E_x = x e_j(e_{k_1}-e_{k_2})^\top,$ and observe that $T+E_x \in SC(P)$ for all such $x.$ Arguing as in the proof of Lemma 2.5 of \cite{kfast}, we find that for each $x$ such that $ x(e_{k_1}-e_{k_2})^\top (I-T)^{\#}e_j \ne 1, \K(T+E_x)= \K(T) + \frac{x}{1- x(e_{k_1}-e_{k_2})^\top (I-T)^{\#}e_j}(e_{k_1}-e_{k_2})^\top ((I-T)^{\#})^2e_j$; in particular this holds whenever $|x|$ is sufficiently small. Further, when $|x|$ is sufficiently small, it is also the case that $T+E_x$ has a single essential class. Since $\K(T)$ is a minimum, it must be the case that  $(e_{k_1}-e_{k_2})^\top ((I-T)^{\#})^2e_j = 0.$ Hence $\K(T) = \K(T+E_x)$ for any $x \in [-t_{j,k_1}, t_{j,k_2}]$ such that $ x(e_{k_1}-e_{k_2})^\top (I-T)^{\#}e_j \ne 1$. Further since $\K(T+E_x)$ is a minimum, it must be the case that $T+E_x$ has a single essential class for any such $x$. If $(e_{k_1}-e_{k_2})^\top (I-T)^{\#}e_j\le 0, $ we  find that $T+E_{t_{j,k_2}}$ satisfies i) and ii) and has fewer positive entries than $T$ does, while if 
	$(e_{k_1}-e_{k_2})^\top (I-T)^{\#}e_j > 0, $ $T+E_{-t_{j,k_1}}$ satisfies i) and ii) and has fewer positive entries than $T$ does. 
\end{proof}

%Lemma 2.5 from digraph paper. 
%There is a matrix having $1$ as a simple eigenvalue that a) attains the %minimum value for Kemeny's constant , and b) has the property that for each %row, at most one unspecified entry is positive. 

\section{$P$ with all specified entries on the diagonal}\label{sec:diag} 

In this section we consider the problem of finding $m(P)$ for the case that the specified entries of the partial stochastic matrix $P$ are all on the diagonal. The following technical result will be useful in the subsequent discussion. Recall that a square matrix is \emph{substochastic} if it is entrywise nonnegative and each row sum is at most $1$. 

%%%%%%need to work on the spacing for partitioned matrices

%%%%need to do something better with this. 
\begin{lemma}\label{lem:trace} 
	\label{trace} Let $T$ be an $n \times n$ substochastic matrix whose  spectral radius is strictly less than $1.$ We have $\trace((I-T)^{-1}) \ge \sum_{k=1}^n \frac{1}{1-t_{k,k}}.$ 
\end{lemma}
\begin{proof}
	We proceed by induction on $n,$ and note that the case $n=1$ is readily established. 
	Suppose now that $n \ge 2$ and that the statement holds for matrices of order $n-1.$ Write $I-T$ in partitioned form as 
	$I-T= \left[\begin{array}{c|c}
		I-\hat{T}&-x  \\  \hline \\ [-9pt]-y^\top & 1-t_{n,n}
	\end{array}\right],
	$
	where $\hat T $ is of order $n-1$ and $x, y \ge 0.$ From the partitioned form of the inverse \cite{hj}, we find that 
\begin{eqnarray*}
&&	\trace((I-T)^{-1}) = \\
&&	\trace \left(\left(	I-\hat{T} -\frac{1}{ 1-t_{n,n}} xy^\top \right)^{-1} + \frac{1}{ 1-t_{n,n} - y^\top (I-\hat{T})^{-1}x} \right).
\end{eqnarray*}
	Hence 
	\begin{eqnarray*}
&&	\trace((I-T)^{-1})) = \\
&&	\trace \left((	I-\hat{T})^{-1}  +\frac{1}{ 1-t_{n,n} - y^\top(	I-\hat{T})^{-1}x } (	I-\hat{T})^{-1}xy^\top(	I-\hat{T})^{-1}\right)\\
	&&+ \frac{1}{ 1-t_{n,n} - y^\top (I-\hat{T})^{-1}x} . 
		\end{eqnarray*}
	We deduce that 
	$\trace((I-T)^{-1})) \ge \trace((I-\hat T)^{-1})) + \frac{1}{ 1-t_{n,n}},$ 
	and applying the induction hypothesis now yields the desired inequality. 
\end{proof}

\begin{lemma}\label{XY}
	Suppose that $T$ is a stochastic matrix of order $n$ with a single essential class. Suppose that  $T =\left[ \begin{array}{c|c} S & \ones - S\ones\\ \hline\\ [-9pt] u^\top&1-\ones^\top u\end{array}\right],$ where the spectral radius of $S$ is less than $1$. 
Then $$\K(T) = \trace((I-S)^{-1})-\frac{1}{1+u^\top (I-S)^{-1}\ones } \trace((I-S)^{-1}\ones u^\top (I-S)^{-1}).$$
\end{lemma}
\begin{proof}
	Write $X=\left[ \begin{array}{c} I-S  \\ \hline\\ [-9pt] -u^\top\end{array}\right],  Y=\begin{bmatrix}
		I &| -\ones
	\end{bmatrix},$ so that $I-T=XY$ is a full rank factorization of $I-T.$ As noted in section \ref{sec:intro}, $\K(T)=\trace((YX)^{-1}).$ Since $YX = I-S+\ones u^\top,$ the conclusion  follows  from the Sherman-Morrison formula \cite{hj}. 
\end{proof}

\begin{example}{\rm{Suppose that for our 
		 partial stochastic matrix $P$ of order $n$, the specified entries  are on the diagonal,  with $p_{j,j}=d_j, j=1, \ldots, n.$ 	
			As a warm-up exercise, we first consider the case that some diagonal entry of $P$ is $1$. Without loss of generality we assume that 
			 $d_1=1$ and $d_j<1, j=2, \ldots, n.$ For any stochastic completion $T$ of $P$ having just one essential class, $T$ is of the form $$T = \left[\begin{array}{c|c} 1&0^\top \\ \hline\\ [-9pt] \ones - \hat T \ones & \hat T
			\end{array}	
			\right]$$ where the spectral radius of $\hat T$ is less than $1,$ and $t_{j,j}=d_j, j=2, \ldots, n.$ 
			Then $\K(T) = \trace((I-\hat T)^{-1}) \ge \sum_{j=2}^n \frac{1}{1-d_j},$ the inequality following from Lemma \ref{lem:trace}. 
			 Evidently equality in the lower bound can be attained by letting $\hat T$ be the diagonal matrix with diagonal entries $d_2, \ldots, d_n$.  Hence $m(P)= \sum_{j=2}^n \frac{1}{1-d_j}.$  
}}	
\end{example}

In the remainder of this section we assume that $d_j<1, j=1, \ldots, n.$ 

%$D=\diag(d_1, \ldots, d_n), \hat D =\diag(d_1, \ldots, d_{n-1}).$ 

\begin{theorem}\label{C} Suppose that $d_j \in [0,1), j=1, \ldots, n.$ 
	Let $C$ be the adjacency matrix for a directed $n$-cycle, and let $T= D+(I-D)C.$ Then 
	$$\K(T) = \frac{1}{2 \sum_{j=1}^n \frac{1}{1-d_j} }\left( \left(\sum_{j=1}^n \frac{1}{1-d_j}\right)^2 - \sum_{j=1}^n \frac{1}{(1-d_j)^2} 
	\right). 
$$
\end{theorem}
\begin{proof}
	Without loss of generality we assume that $C$ corresponds to the cycle $1 \rightarrow 2 \ldots \rightarrow n \rightarrow 1.$ 
	We apply Lemma \ref{XY}, and note that in the notation of that lemma, 
	the submatrix of $T$ consisting of its first $n-1$ columns can be written as $\left[ \begin{array}{c}
		\hat D+(I-\hat D)N\\ \hline\\ [-9pt]  (1-d_n)e_1^\top
	\end{array}\right],$ where $N$ is the matrix of order $n-1$ whose first superdiagonal consists of ones, and where  all other entries are zero, and where $\hat D$ is the leading $(n-1)\times(n-1)$ principal submatrix of $D$.  It now follows that $\K(T)$ is the trace of the matrix 
\begin{eqnarray*}
	&&
	 (I-N)^{-1} (I-\hat D)^{-1}
	 - \left(
	 \frac{1-d_n}{1+(1-d_n)e_1^\top (I-N)^{-1} (I-\hat D)^{-1}\ones }\right) \times \\
&&	  (I-N)^{-1} (I-\hat D)^{-1}\ones   e_1^\top (I-N)^{-1} (I-\hat D)^{-1}.
\end{eqnarray*}
Observe that $(I-N)^{-1}$ is the upper triangular matrix of order $n-1$ whose entries on and above the diagonal are ones. From that observation we find the following: 
a)  $\trace((I-N)^{-1} (I-\hat D)^{-1} ) = \sum_{j=1}^{n-1} \frac{1}{1-d_j}$;
b) $\frac{1-d_n}{1+(1-d_n)e_1^\top (I-N)^{-1} (I-\hat D)^{-1}\ones } = \frac{1}{\sum_{j=1}^{n} \frac{1}{1-d_j}}$.  Recalling that for vectors $u,v\in \reals^{n-1}, \trace(uv^\top)=v^\top u,$ we find that 
$\trace((I-N)^{-1} (I-\hat D)^{-1}\ones e_1^\top (I-N)^{-1} (I-\hat D)^{-1})=$ $  \sum_{j=1}^{n-1} \frac{1}{1-d_j} \sum_{\ell =j}^{n-1} \frac{1}{1-d_\ell}.$ 

Next, we note that $ \sum_{j=1}^{n-1} \frac{1}{1-d_j} \sum_{\ell =j}^{n-1} \frac{1}{1-d_\ell}$ can be written as $$\frac{1}{2}\left(\left (\sum_{j=1}^{n-1}\frac{1}{1-d_j}  \right)^2 + \sum_{j=1}^{n-1}\frac{1}{(1-d_j)^2}   \right).$$ Assembling these observations, we find that 
$$\K(T) = \sum_{j=1}^{n-1}\frac{1}{1-d_j} - \frac{(\sum_{j=1}^{n-1}\frac{1}{1-d_j}  )^2 + \sum_{j=1}^{n-1}\frac{1}{(1-d_j)^2}  }{2(\sum_{j=1}^{n} \frac{1}{1-d_j})}.$$
Simplifying that expression now yields 
\begin{flalign*}
&	\K(T)  =% \frac{1}{2(\sum_{j=1}^{n} \frac{1}{1-d_j})} \left( \left(\frac{2}{1-d_n} + 2\sum_{j=1}^{n-1} \frac{1}{1-d_j}\right)\sum_{j=1}^{n-1} \frac{1}{1-d_j} - \left(\sum_{j=1}^{n-1}\frac{1}{1-d_j}  \right)^2 \right. \\ && \left. -\sum_{j=1}^{n-1}\frac{1}{(1-d_j)^2}	
%	\right) \\
	%&=& 
	\frac{1}{2(\sum_{j=1}^{n} \frac{1}{1-d_j})} \left(
	\left(\sum_{j=1}^{n-1}\frac{1}{1-d_j}  \right)^2 + \frac{2}{1-d_n} \sum_{j=1}^{n-1} \frac{1}{1-d_j} -\sum_{j=1}^{n-1}\frac{1}{(1-d_j)^2}
	 \right) \\ 
	 &=  \frac{1}{2 \sum_{j=1}^n \frac{1}{1-d_j} }\left( \left(\sum_{j=1}^n \frac{1}{1-d_j}\right)^2 - \sum_{j=1}^n \frac{1}{(1-d_j)^2} 
	 \right).
\end{flalign*}
\end{proof}
 
 \begin{lemma}\label{xj} 
 	Suppose that $n \ge 2$ and  $x_j>0, j=1, \ldots, n,$ and consider 
 	$$s_k = \frac{(\sum_{j=1}^k x_j)^2- \sum_{j=1}^k x_j^2}{2\sum_{j=1}^k x_j} + \sum_{j=k+1}^n x_j.$$ We have $s_1 > s_2 > \ldots > s_n.$
 \end{lemma}
\begin{proof}
	Suppose that $1\le k\le n-1.$ Observe that 
	\begin{eqnarray*}
	&&	s_k-s_{k+1} =\\
		 && -\frac{1}{2}x_{k+1} + x_{k+1} + \frac{\sum_{j=1}^{k+1} x_j^2}{2\sum_{j=1}^{k+1} x_j} - 
		\frac{\sum_{j=1}^{k} x_j^2}{2\sum_{j=1}^{k} x_j}=\\
		&& \frac{1}{2}\left(x_{k+1} + \frac{\sum_{j=1}^{k} x_j  \sum_{j=1}^{k+1} x_j^2 -\sum_{j=1}^{k+1} x_j \sum_{j=1}^{k} x_j^2 }{\sum_{j=1}^{k} x_j\sum_{j=1}^{k+1} x_j} 	\right) =\\
		&& \frac{1}{2\sum_{j=1}^{k} x_j\sum_{j=1}^{k+1} x_j}\left( x_{k+1}^2\sum_{j=1}^{k} x_j + x_{k+1} \sum_{j=1}^{k} x_j \sum_{j=1}^{k+1} x_j - x_{k+1}\sum_{j=1}^{k} x_j^2 		\right) .
	\end{eqnarray*}
Since $\sum_{j=1}^{k} x_j \sum_{j=1}^{k+1} x_j >\sum_{j=1}^{k} x_j^2 ,$ we deduce that $s_k > s_{k+1}.$ 	
\end{proof}

%**need to rephrase thm as a completion problem result.***

\begin{theorem}\label{diagmain} 
	Suppose that $d_j \in [0,1), j=1, \ldots, n$ and that $P$ is an $n \times n$ partial stochastic matrix with $p_{j,j}=d_j, j=1, \ldots, n$ and all other entries unspecified. 
	Then 
	$$
	m(P)=  
	\frac{1}{2 \sum_{j=1}^n \frac{1}{1-d_j} }\left( \left(\sum_{j=1}^n \frac{1}{1-d_j}\right)^2 - \sum_{j=1}^n \frac{1}{(1-d_j)^2} 
	\right).   
$$
\end{theorem}
\begin{proof}
	Suppose that $T$ is a completion of $P$ such that $\K(T)=m(P).$  By Proposition \ref{entries}, we may assume without loss of generality that $T$ has precisely one nonzero off-diagonal entry in each row, and a single essential class of indices.  
	In the directed graph of $T,$ for each vertex $j$ there is precisely one  vertex $\ell \ne j$ such that $j \rightarrow \ell.$ Hence  this directed graph has 
	at least one cycle of length greater than $1$, and if it were to have more than one 
 	such cycle then there would be more than one essential class. 
	  We deduce that the directed graph of $T$ has precisely one cycle, of length $k$ say, and without loss of generality it is on vertices $1, \ldots, k.$ 
	
	%%rephrase this first bit
 If $k<n,$ then $\K(T) = \frac{(\sum_{j=1}^k \frac{1}{1-d_j})^2- \sum_{j=1}^k \frac{1}{(1-d_j)^2}}{2\sum_{j=1}^k \frac{1}{1-d_j}} + \sum_{j=k+1}^n \frac{1}{1-d_j}.$ Appealing to Lemma \ref{xj} (with $x_j =\frac{1}{1-d_j}$) we see that $T$ cannot minimize Kemeny's constant over $SC(P)$. Hence  $k=n,$ and the conclusion now follows from Theorem \ref{C}.
\end{proof}

\begin{remark}\label{incrj0} 
	\rm{Fix an index $j_0 =1, \ldots, n,$ and consider the expression $$f\equiv \frac{1}{ \sum_{j=1}^n \frac{1}{1-d_j} }\left( \left(\sum_{j=1}^n \frac{1}{1-d_j}\right)^2 - \sum_{j=1}^n \frac{1}{(1-d_j)^2} \right) $$ as a function of $d_{j_0}.$ A straightforward computation shows that $\frac{\partial  f}{\partial  d_{j_0}}$ has the same sign as $\left(\sum_{j \ne j_0} \frac{1}{1-d_j}\right)^2 + \sum_{j \ne j_0} \frac{1}{(1-d_j)^2}.$ In particular, $f$ is increasing in each $d_j$. 
}
\end{remark}

\begin{corollary}
	Suppose that $d_j \in [0,1), j=1, \ldots, k.$ Suppose that $P$ is a partial stochastic matrix of order $n>k$ with $p_{j,j}=d_j, j=1, \ldots, k,$ and all other entries unspecified. 
	Then $	m(P) =  $ 
$$		\frac{1}{2 (\sum_{j=1}^k \frac{1}{1-d_j} +n-k) }\left( \left(\sum_{j=1}^k \frac{1}{1-d_j} + n-k\right)^2 - \sum_{j=1}^k \frac{1}{(1-d_j)^2} -(n-k) 
		\right).   			$$ 
\end{corollary}
\begin{proof}
Set $d_{k+1}, \ldots, d_n$	 all equal to $0$. From Theorem \ref{diagmain} and Remark \ref{incrj0} we find that for any stochastic matrix $T$ such that $t_{j,j}=d_j, j=1, \ldots, k,$ $\K(T)\ge  \frac{1}{2 (\sum_{j=1}^k \frac{1}{1-d_j} +n-k) }\left( \left(\sum_{j=1}^k \frac{1}{1-d_j} + n-k\right)^2 - \sum_{j=1}^k \frac{1}{(1-d_j)^2 }-(n-k) 
\right).$ Let $C$ denote a cyclic permutation matrix of order $n$. 
Setting $D=\diag(d_1, \ldots, d_n)$ and $T_0=D+(I-D)C,$ we find from Theorem \ref{C} that $\K(T_0) =$ $$ \frac{1}{2 (\sum_{j=1}^k \frac{1}{1-d_j} +n-k) }\left( \left(\sum_{j=1}^k \frac{1}{1-d_j} + n-k\right)^2 - \sum_{j=1}^k \frac{1}{(1-d_j)^2 }-(n-k) 
\right).$$ 
\end{proof}

%bound increasing in each $d_j$, leads to the min when only part of the %diagonal is specified. 

%Need to work on the characterization of equality. 

Our next goal is to characterize the matrices yielding the minimum value $m(P)$ in Theorem \ref{diagmain}. We first establish some  technical results that will assist in that characterization. 
The following formula appears in the proof of Lemma 2.5 in \cite{kfast}. 

\begin{lemma}\label{gen_der} 
	Suppose that $T$ is a stochastic matrix with one essential class, and for each $x \in [-t_{i,p}, t_{i,q}]$, set $S_x=T+ xe_i(e_p-e_q)^\top.$ Then for any $x \in [-t_{i,p}, t_{i,q}]$ such that $1-x(e_p-e_q)^\top(I-T)^{\#}e_i \ne 0,$ $\K(S_x)=\K(T) + \frac{x}{1-x(e_p-e_q)^\top(I-T)^{\#}e_i} (e_p-e_q)^\top ((I-T)^{\#})^2 e_i.$
\end{lemma}

\begin{lemma}\label{cyc:der} 
	Suppose that $d_j \in [0,1), j=1, \ldots, n$, let $D=\diag(d_1, \ldots, d_n)$ and let $C$ be the permutation matrix corresponding to the cyclic permutation  $1 \rightarrow 2 \rightarrow \ldots \rightarrow n-1 \rightarrow n \rightarrow 1.$ Let $T=D+(I-D)C$. Fix an index $k=2, \ldots, n-1,$ and for each $x\in [0,1-d_n],$ let $T_x=T+xe_n(e_k-e_1)^\top.$ Then 
	$\K(T_x)>\K(T)$ for all $x\in (0,1-d_n].$ 
	%and further 
	%$\frac{d \K(T_x)}{dx}|_{x=0}>0.$ (Here we interpret this as a derivative %from the right.)
\end{lemma}
\begin{proof}
	Here we employ the notation of Theorem \ref{C}. Let $X=\left[ \begin{array}{c}
		 (I-\hat D)(I-N)\\ \hline\\ [-9pt]  -(1-d_n)e_1^\top
	\end{array}\right],$ $Y=\left[\begin{array}{c|c}I&-\ones\end{array}\right].$
	Then $I-T=XY$  so that $(I-T)^{\#} = X(YX)^{-2}Y. $

	We have $Ye_n = -\ones,$ and it now follows that 
	\begin{equation} \label{dery} 
			(YX)^{-1}Y e_n=\frac{-1}{1+(1-d_n)e_1^\top(I-N)^{-1}(I-\hat D)^{-1}\ones}(I-N)^{-1}(I-\hat D)^{-1} \ones.
	\end{equation}
Also, $(e_k-e_1)^\top X=((1-d_k)e_k-(1-d_1)e_1)^\top(I-N),$ and it follows that  
\begin{equation}\label{derx} 
	(e_k-e_1)^\top X(YX)^{-1} =(e_k-e_1)^\top.
\end{equation}
We find from \eqref{dery} and \eqref{derx}  that
\begin{eqnarray*}
&& 1-x(e_k-e_1)^\top(I-T)^{\#}e_n = 
   1-x(e_k-e_1)^\top X(YX)^{-1}(YX)^{-1}Y e_n = \\
  &&  1+ \frac{x}{1+(1-d_n)e_1^\top(I-N)^{-1}(I-\hat D)^{-1}\ones}(e_k-e_1)^\top(I-N)^{-1}(I-\hat D)^{-1} \ones. 
\end{eqnarray*}
Hence 
	$1-x(e_k-e_1)^\top(I-T)^{\#}e_n = 1+\frac{x}{1-d_n}\frac{\sum_{j=1}^{k-1} \frac{1}{1-d_j}}{\sum_{j=1}^{n} \frac{1}{1-d_j}}>0.$ In particular, we find from Lemma \ref{gen_der} that 
$
		\K(T_x) = \K(T)+ 
	\frac{x(1-d_n)\sum_{j=1}^{n} \frac{1}{1-d_j}}{ (1-d_n)\sum_{j=1}^{n} \frac{1}{1-d_j} +x \sum_{j=1}^{k-1} \frac{1}{1-d_j} } (e_k-e_1)^\top ((I-T)^{\#})^{2} e_n. 
$
	It remains to compute $(e_k-e_1)^\top ((I-T)^{\#})^{2} e_n.$ 
	
We have 
	 $((I-T)^{\#})^{2} = X(YX)^{-3}Y.$ From \eqref{dery} and \eqref{derx}, 
	% We want to compute $(e_k-e_1)^\top ((I-T)^{\#})^{2} e_n,$ where $2 \le k \le n-1.$ We have $Ye_n = -\ones,$ and it now follows that $(YX)^{-1}Y e_n=\frac{-1}{1+(1-d_n)e_1^\top(I-N)^{-1}(I-\hat D)^{-1}\ones}(I-N)^{-1}(I-\hat D)^{-1} \ones.$ 
	%Also, $(e_k-e_1)^\top X=(((1-d_k)e_k-(1-d_1)e_1)^\top)(I-N),$ and it follows that  $(e_k-e_1)^\top X(YX)^{-1} =(e_k-e_1)^\top.$ 
	we deduce that $(e_k-e_1)^\top X(YX)^{-3}Y e_n$ has the same sign as $(e_1-e_k)^\top (YX)^{-1} (I-N)^{-1}(I-\hat D)^{-1} \ones.$ 
	
	Observe that for each $j=1, \ldots, n-1, $ the $j$-th entry of $(I-N)^{-1}(I-\hat D)^{-1} \ones$ is $\sum_{\ell=j}^{n-1} \frac{1}{1-d_\ell}.$ 
	We also have $$(e_1-e_k)^\top (YX)^{-1} =   \sum_{j=1}^{k-1}\frac{1}{1-d_j}e_j^\top - \frac{ \sum_{\ell=1}^{k-1}\frac{1}{1-d_\ell}}{ \sum_{\ell=1}^{n}\frac{1}{1-d_\ell}}\sum_{j=1}^{n-1}\frac{1}{1-d_j}e_j^\top.$$  
	We thus find that $(e_k-e_1)^\top ((I-T)^{\#})^{2} e_n$ has the same sign as 	\begin{eqnarray}
		\sum_{\ell=1}^{n}\frac{1}{1-d_\ell}\left(\sum_{j=1}^{k-1}\frac{1}{1-d_j}  
		\sum_{\ell=j}^{n-1} \frac{1}{1-d_\ell}\right)
		- \sum_{\ell=1}^{k-1}\frac{1}{1-d_\ell} \left(\sum_{j=1}^{n-1}\frac{1}{1-d_j}  
		\sum_{\ell=j}^{n-1} \frac{1}{1-d_\ell}\right).
		\label{eq:der}
\end{eqnarray}
We claim that \eqref{eq:der} is strictly positive.

In order to establish the claim, we  rewrite \eqref{eq:der} as follows: 
	\begin{flalign*}
&	\sum_{\ell=1}^{n}\frac{1}{1-d_\ell}\left(\sum_{j=1}^{k-1}\frac{1}{1-d_j}  
		\sum_{\ell=k}^{n-1} \frac{1}{1-d_\ell} 
		+\sum_{j=1}^{k-1}\frac{1}{1-d_j}  
		\sum_{\ell=j}^{k-1} \frac{1}{1-d_\ell} 
		\right)\\
&		- \sum_{\ell=1}^{k-1}\frac{1}{1-d_\ell} \left( 
		\frac{1}{2} 	\left(\sum_{\ell=1}^{n-1}\frac{1}{1-d_\ell}\right)^2 + 
		\frac{1}{2}\sum_{\ell=1}^{n-1}\frac{1}{(1-d_\ell)^2}
		\right) \\
&		=	\sum_{\ell=1}^{n}\frac{1}{1-d_\ell}\sum_{j=1}^{k-1}\frac{1}{1-d_j}  
		\sum_{\ell=k}^{n-1} \frac{1}{1-d_\ell} 
		+ \sum_{\ell=1}^{n}\frac{1}{1-d_\ell} \left( 
		\frac{1}{2} 	\left(\sum_{\ell=1}^{k-1}\frac{1}{1-d_\ell}\right)^2 + 
		\frac{1}{2}\sum_{\ell=1}^{k-1}\frac{1}{(1-d_\ell)^2}
		\right) \\
&		- \sum_{\ell=1}^{k-1}\frac{1}{1-d_\ell} \left( 
		\frac{1}{2} 	\left(\sum_{\ell=1}^{n-1}\frac{1}{1-d_\ell}\right)^2 + 
		\frac{1}{2}\sum_{\ell=1}^{n-1}\frac{1}{(1-d_\ell)^2}
		\right) \end{flalign*}\begin{flalign*}
&		=	\sum_{\ell=1}^{n}\frac{1}{1-d_\ell}\sum_{j=1}^{k-1}\frac{1}{1-d_j}  
		\sum_{\ell=k}^{n-1} \frac{1}{1-d_\ell} 
	+ \frac{1}{2}  \sum_{\ell=1}^{n}\frac{1}{1-d_\ell} 
		\left(\sum_{\ell=1}^{k-1}\frac{1}{1-d_\ell}\right)^2 \\
&		+\frac{1}{2}  \sum_{\ell=1}^{n}\frac{1}{1-d_\ell} \sum_{\ell=1}^{k-1}\frac{1}{(1-d_\ell)^2}		-\frac{1}{2} \sum_{\ell=1}^{k-1}\frac{1}{1-d_\ell}\left(\sum_{\ell=1}^{k-1}\frac{1}{1-d_\ell}\right)^2 
		-\frac{1}{2} \sum_{\ell=1}^{k-1}\frac{1}{1-d_\ell}\left(\sum_{\ell=k}^{n-1}\frac{1}{1-d_\ell}\right)^2\\
&	-\sum_{\ell=1}^{k-1}\frac{1}{1-d_\ell}\sum_{\ell=1}^{k-1}\frac{1}{1-d_\ell}\sum_{\ell=k}^{n-1}\frac{1}{1-d_\ell} 
	-\frac{1}{2}\sum_{\ell=1}^{k-1}\frac{1}{1-d_\ell}\sum_{\ell=1}^{n-1}\frac{1}{(1-d_\ell)^2} \\
&		= \frac{1}{2} \sum_{j=k}^n \frac{1}{1-d_j} \left(\sum_{\ell=1}^{k-1}\frac{1}{1-d_\ell} \right)^2 + 
		\sum_{\ell=1}^{k-1}\frac{1}{1-d_\ell}\sum_{\ell=k}^{n-1}\frac{1}{1-d_\ell}\sum_{\ell=k}^{n}\frac{1}{1-d_\ell} \\
&		-\frac{1}{2}\sum_{\ell=1}^{k-1}\frac{1}{1-d_\ell} \left(\sum_{\ell=k}^{n-1}\frac{1}{1-d_\ell}\right)^2 
		+\frac{1}{2} \sum_{\ell=k}^{n}\frac{1}{1-d_\ell}\sum_{\ell=1}^{k-1}\frac{1}{(1-d_\ell)^2}   \\
&	-\frac{1}{2}\sum_{\ell=1}^{k-1}\frac{1}{1-d_\ell}\sum_{\ell=k}^{n-1}\frac{1}{(1-d_\ell)^2}.	 
	\end{flalign*}
	Inspecting this final expression, we see that \eqref{eq:der} is strictly positive. The conclusion now follows readily. 
\end{proof}

\begin{theorem}\label{thm:diag_eq} 
	Suppose that $d_j \in [0,1), j=1, \ldots, n$ and that $P$ is an $n \times n$ partial stochastic matrix with $p_{j,j}=d_j, j=1, \ldots, n$ and all other entries unspecified. Let $D=\diag(d_1, \ldots, d_n).$ Then $T \in SC(P)$ with $\K(T)=m(P)$ if and only if $T=D+(I-D)C,$  where $C$ is a cyclic permutation matrix for a cycle of length $n$.
\end{theorem}
\begin{proof}
	Suppose that $T\in SC(P)$ and that $T$ cannot be written as $D+(I-D)C,$  for any $n$-cyclic permutation matrix $C$. Let the number of positive off-diagonal entries in $T$ be $\ell$. We claim that $\K(T)> m(P),$ and 
	 proceed by induction on $\ell$. 
	
	If $\ell=n,$ then from the proof of Theorem \ref{diagmain}, we find that for some $k<n,$ $\K(T)=$ $ \frac{(\sum_{j=1}^k \frac{1}{1-d_j})^2- \sum_{j=1}^k \frac{1}{(1-d_j)^2}}{2\sum_{j=1}^k \frac{1}{1-d_j}} + \sum_{j=k+1}^n \frac{1}{1-d_j}.$ The conclusion then follows from Lemma \ref{xj}. 
	
	Suppose now that $\ell>n$ (here we mimic the approach of  Proposition  \ref{entries}). Then there are 	
	indices $j, k_1, k_2$ such that $t_{j,k_1}, t_{j,k_2}>0.$  If $(e_{k_1}-e_{k_2})^\top ((I-T)^{\#})^2e_j \ne 0$ then we may perturb $T$ slightly in the $j$-th row to produce a matrix $\hat T \in SC(P)$ such that $\K(\hat T)<\K(T).$ On the other hand, if $(e_{k_1}-e_{k_2})^\top ((I-T)^{\#})^2e_j = 0,$ then there is a matrix $\tilde T \in SC(P)$ such that $\K(\tilde T)=\K(T)$ and $\tilde T$ has $\ell-1$ positive off-diagonal entries. If $\tilde T$ is not of the form  $D+(I-D)C$  for any $n$-cyclic permutation matrix $C$ then $\K(\tilde T)> m(P)$ by the induction hypothesis. Suppose now that $\tilde T=D+(I-D)C $ for some $n$-cyclic permutation matrix $C$, and without loss of generality $C$ corresponds to the permutation $1 \rightarrow 2 \rightarrow \ldots \rightarrow n \rightarrow 1.$ In that case $T$ is permutationally similar to a matrix $T_x = D+(I-D)C +	xe_n(e_k-e_1)^\top$ for some $x>0.$ Then $\K(T)=\K(T_x)>m(P),$ the strict inequality following from Lemma \ref{cyc:der}.

\end{proof}

%Claim 2: The matrices of the form $D+(I-D)C$ are the only minimizers. \\
%Suppose not. Then there is another matrix $T_0$ that minimizes $\K$ and has more than $n$ positive off-diagonal entries.  the technique of Proposition  \ref{entries}, we find that there is a sequence of matrices $T_0, T_1, \ldots, T_k$ such that the values of $\K(T_j)$ are all the same (and hence minimal), and in addition each $T_j$ has one more zero than $T_{j-1}$ does, $T_j-T_{j-1}\equiv E$ has exactly two nonzero entries, both in the same row, and $\K(T_j+xE)=\K(T_{j+1})$ for all $x $ in the appropriate interval. 

%From the above, if there were such a $T_0,$ then wlog there is a rank one perturbation of $D+(I-D)C$ of the form $D+(I-D)C + x e_n(e_k-e_1)^\top$ such that $\K(D+(I-D)C )= \K(D+(I-D)C + x e_n(e_k-e_1)^\top)$ for all sufficiently small $x$. This last contradicts Claim 1. 

\section{$P$ with all specified entries in one row}\label{sec:row}

%**possible remark on $\rho_j$ convention. convention of $r$ as given with %concomb constraints\\

%******need to emphasize the reordering, maybe after the k=n thm

In this section we consider the case that the $n \times n$ partial stochastic matrix $P$ has $n$ specified entries, all lying in a common row. By applying a permutation similarity if necessary, we may assume 
without loss of generality that the specified entries are in row $n$, and take the values $r_0, r_1, \ldots, r_{n-1}$, with $r_0$ on the diagonal of $P$, with  the remaining values occupying positions $(n,1), \ldots, (n,n-1).$ Throughout this section we assume that $r_j\ge 0, j=0, \ldots, n-1$ and $\sum_{j=0}^{n-1}r_j=1.$ 

As in section \ref{prelim}, we consider a $T\in SC(P)$ that minimizes Kemeny's constant, and  in view of Proposition \ref{entries}, we focus on  the case that the first $n-1$ rows of $T$ are $(0,1)$. That leads to two possibilities for the directed graph of $T$: either there is a cycle in the subgraph induced by vertices $1, \ldots, n-1,$ or every cycle passes through vertex $n$. The next two results address those cases. 

\begin{proposition}\label{row-cycle} 
	Suppose that the first $n-1$ rows of the stochastic matrix  $T$ are $(0,1)$,  the directed graph of $T$ has a cycle on vertices $1, \ldots, k$, that $T$ has a single essential class, and that the last row of $T$ is specified, with diagonal entry $r_0.$ Then $\K(T) = \frac{2n-k-3}{2}+\frac{1}{1-r_0}.$ In particular, $\K(T) \ge \frac{n-2}{2}+\frac{1}{1-r_0},$ with equality if $k=n-1.$ 
\end{proposition}
\begin{proof}
	Write $T$ as $T=\left[\begin{array}{c|c|c}C&0&0\\ \hline X &N&0\\ \hline\\ [-9pt] u^\top &v^\top &r_0\end{array}		
		\right],$  
		where $C$ is the adjacency matrix of the $k$-cycle,  $N$ is a $(0,1)$ matrix of order $n-k-1$, and $u^\top, v^\top $ are vectors whose entries comprise $r_1, \ldots, r_{n-1}$. The directed graph of $N$ cannot contain a cycle, otherwise the fact that $N$ is $(0,1)$ implies that $T$ has more than one essential class. Consequently $N$ is  nilpotent. 
		
		We now find that $\K(T)= \K(C) + \trace((I-N)^{-1})+ \frac{1}{1-r_0} = 
		\frac{k-1}{2}+n-k-1+\frac{1}{1-r_0}.$ The conclusion follows. 
\end{proof}

\begin{proposition}\label{row-nilp}
Suppose that the first $n-1$ rows of the stochastic matrix $T$ are $(0,1)$ and that every cycle in the directed graph of $T$  goes through  vertex $n$. 
Suppose that in the last row of $T$, the entries $r_1, \ldots, r_{n-1}$ appear in positions $1, \ldots, n-1$ (in some order) and that the $(n,n)$ entry is $r_0$.  Partition the vertices of the directed graph associated with $T$  into sets $A_0=\{n\}, A_1, \ldots, A_{d-1}$ where $A_j$ is the set of vertices at distance $j$ from $n$, and $d-1$ is the maximum distance from $n$. For each $j=0, \ldots, d-1,$ let $\tilde r_j=\sum_{k \in A_j} r_k.$ 
Then $\K(T)=$ $ n-1-\frac{\sum_{j=0}^{d-1}j(j+1)\tilde r_j}{2\sum_{j=0}^{d-1}(j+1)\tilde r_j}.$ 
\end{proposition}
\begin{proof}
	We proceed by induction on $n$ and note that the case $n=2$ is straightforward. Henceforth we suppose that $n \ge 3.$ We subdivide the rest of the proof into two cases. \\
	
	\noindent \textit{Case 1}, the first $n-1$ rows of $T$ are distinct. \\  
	In this case, we have $d=n,$ and each $A_j$ consists of a single vertex. Without loss of generality we may 
		write $T$ as 
		$\left[ \begin{array}{c|c} N&e_{n-1}\\ \hline \\ [-9pt] u^\top & \tilde r_0
		\end{array}\right], $  
		where $N$ is the nilpotent matrix of order $n-1$ with ones on the superdiagonal and zeros elsewhere, and where the entries in $u^\top $ correspond to  $r_1, \ldots, r_{n-1}.$ Indeed, in the notation of the statement, $u_j = \tilde r_{n-j}, j=1, \ldots, n-1.$ 
		
	%	Set $X=\begin{bmatrix} I-N \\ \hline -u^\top\end{bmatrix}, Y=\begin{bmatrix}
	%		I &| -\ones
	%	\end{bmatrix}.$ 
	Applying Lemma \ref{XY} we find that 
	\begin{eqnarray*}
	\K(T)&=&  \trace((I-N)^{-1}) - \frac{1}{1+u^\top (I-N)^{-1} \ones }\trace( (I-N)^{-1}\ones u^\top (I-N)^{-1}   )  \\
&=&	n-1 - \frac{1}{1+u^\top (I-N)^{-1} \ones }  u^\top (I-N)^{-1} (I-N)^{-1}\ones. 
	\end{eqnarray*}
	The desired expression now follows from a computation. \\
		
		\noindent \textit{Case 2}, among the first $n-1$ rows of $T$, two rows 
		are the same. \\
		For concreteness, but without loss of generality, we may assume that the first two rows of $T$ are equal to $e_3^\top.$ We may then write 
		$T$ as $T= \left[ \begin{array}{c|c}0^\top&e_1^\top\\ 0^\top &e_1^\top \\ \hline \\ [-9pt]  T_1 &  T_2\end{array} \right],$ 
		where: $ T_1$ is $(n-2)\times 2$,  $T_2$ is $(n-2)\times (n-2), $ the first $n-3$ rows of $\left[ \begin{array}{c|c}T_1&T_2\end{array}\right]$ are $(0,1),$ and the last row of  $\left[ \begin{array}{c|c}T_1&T_2\end{array}\right]$ consists of the entries $r_1, \ldots, r_{n-1}$ in the first $n-1$ positions and $r_0$ in the last position.

		We may factor $T$ as $UV$ where $U=
		\left[\begin{array}{c|c} \ones & \ones 0^\top  \\ \hline \\ [-9pt] 0&I
		 \end{array}\right], V=
	 \left[ \begin{array}{c|c}0^\top&e_1^\top\\ \hline \\ [-9pt]  T_1 &  T_2 \end{array} \right].$ Here the rows of $U$ are partitioned as $2$ rows/$(n-2)$ rows, and the columns are partitioned as $1$ column/$(n-2)$ columns, while the rows of $V$ are partitioned as $1$ row/$(n-2)$ rows, and the columns are partitioned as $2$ columns/$(n-2)$ columns. 
	 The eigenvalues of $T$ consist of $0,$ along with the eigenvalues of the $(n-1)\times (n-1)$ matrix $VU=\left[ \begin{array}{c|c} 0&e_1^\top \\ \hline \\ [-9pt]  T_1\ones & T_2
	 \end{array} \right].$ 
 Observe that since the first $n-3$ rows of  $ T_1 $ are $(0,1)$ with at most one $1$ in each row, $ T_1\ones $ is necessarily $(0,1)$ in its first $n-3$ positions.  
		Observe also that $\K(T)=\K(UV)=1+\K(VU).$ 
		
		Note that the induction hypothesis applies to $VU$, and in particular that in the directed graph of $T,$ vertices $1$ and $2$ of $T$ (which correspond to vertex $1$ of $VU$) are at the same distance from $n$. Applying the induction hypothesis to $VU$, we find that $\K(VU) =  
		n-2-\frac{\sum_{j=0}^{d-1}j(j+1)\tilde r_j}{2\sum_{j=0}^{d-1}(j+1)\tilde r_j}.$ The expression for $\K(T)$ now follows. 
\end{proof}

From Proposition \ref{row-nilp}, in order to minimize Kemeny's constant when one row of $P$ is specified, we would like to maximize $\frac{\sum_{j=0}^{d-1}j(j+1)\tilde r_j}{\sum_{j=0}^{d-1}(j+1)\tilde r_j}$
over the various choices of $d$ and the subsets $A_1, \ldots, A_{d-1}.$ The next result is useful in that regard. 

\begin{lemma}\label{per}
	Suppose that $d<n$ and that we have the partition $A_0, A_1, \ldots, A_{d-1},$ where $|A_{j_0}|\ge 2,$ say with $\ell \in A_{j_0}.$ 
	Consider the partition $ A_j', j=0, \ldots , d$ with $ A_j'=A_j, j \in \{0, \ldots, d-1\} \setminus\{j_0\},  A_{j_0}'= A_{j_0}\setminus \{\ell\},  A_d'=\{\ell\}.$ Let $g=\frac{\sum_{j=0}^{d-1}j(j+1)\tilde r_j}{\sum_{j=0}^{d-1}(j+1)\tilde r_j}$ be associated with the original partition $A_0, A_1, \ldots, A_{d-1},$ and let $h = \frac{\sum_{j=0}^{d}j(j+1)\tilde r_j'}{\sum_{j=0}^{d}(j+1)\tilde r_j'}$ be associated with the modified partition $ A_0',  \ldots,  A_{d}'.$ 
	Then $h \ge g.$  
\end{lemma}
\begin{proof}
	Observe that $\tilde r_j' =\tilde r_j$ for $j=0, \ldots, d-1, j\ne j_0,$ while $\tilde r_{j_0}'=\tilde{r}_{j_0} -r_\ell, \tilde r_d=r_\ell.$ It now follows that $h$ can be written as 
	$h = \frac{\sum_{j=0}^{d-1}j(j+1)\tilde r_j + (d-j_0)(d+j_0+1)r_\ell}{\sum_{j=0}^{d-1}(j+1)\tilde r_j + (d-j_0)r_\ell}.$ A computation now reveals that $h-g$ has the same sign as $r_\ell(d-j_0)((d+j_0+1)\sum_{j=0}^{d-1}(j+1)\tilde r_j- \sum_{j=0}^{d-1}j(j+1)\tilde r_j  ) = r_\ell(d-j_0) \sum_{j=0}^{d-1}(j+1)\tilde r_j(d+j_0+1-j),$ which is clearly nonnegative. 
\end{proof}

\begin{remark}{\rm{From Lemma \ref{per}, we see that in order to maximize $\frac{\sum_{j=0}^{d-1}j(j+1)\tilde r_j}{\sum_{j=0}^{d-1}(j+1)\tilde r_j},$ it suffices to consider the case that $d=n,$ and that each $A_j$ consists of a single index. It follows that the corresponding stochastic matrices are of the form  
$$\begin{bmatrix}
	0&1&0& \ldots &0&0\\
	0&0&1&0& \ldots &0\\
	\vdots & & & \ddots & & \vdots\\
		0&0& \ldots &0&1&0\\
	0&0& \ldots &0&0&1\\
	r_{\sigma(n-1)} & r_{\sigma{(n-2)}}& \ldots &r_{\sigma(2)}&r_{\sigma(1)}&r_0
\end{bmatrix}	$$
for  permutations $\sigma \in S_{n-1}.$ 
} }
\end{remark}

Here is one of the main results in this section. 

\begin{theorem} \label{row-main} Suppose that $P$ is an $n \times n$ partial  stochastic matrix where row $n$ has been specified as $p_{n,j}=r_{n-j}, j=1, \ldots, n,$ and all remaining entries are unspecified. Then 
	$$m(P) = \min \left\{ \frac{n-2}{2}+\frac{1}{1-r_0},  n-1-\frac{\sum_{j=1}^{n-1}j(j+1) r_{\sigma(j)}}{2(1+\sum_{j=1}^{n-1}j r_{\sigma(j)})} \right\},$$ where the minimum is taken over all permutations $\sigma $ of $\{1, \ldots, n-1\}.$ 
\end{theorem}
\begin{proof} 
	Adopting the notation of Proposition \ref{row-nilp}, it follows 
	from Lemma \ref{per}  that in order to maximize $\frac{\sum_{j=0}^{d-1}j(j+1)\tilde r_j}{\sum_{j=0}^{d-1}(j+1)\tilde r_j}$ it suffices to consider $d=n, A_0=\{n\},$ and $A_j=\{\sigma(j)\}$ for some   permutation $\sigma $ of $\{1, \ldots, n-1\}.$  From the fact that $(\tilde r_0=)r_0=1-\sum_{k=1}^{n-1} r_{\sigma(k)},$ it follows that $\frac{\sum_{j=0}^{n-1}j(j+1)\tilde r_j}{\sum_{j=0}^{n-1}(j+1)\tilde r_j} = \frac{\sum_{j=1}^{n-1}j(j+1) r_{\sigma(j)}}{(1+\sum_{j=1}^{n-1}j r_{\sigma(j)})}.$ 
	The conclusion now follows from Propositions \ref{row-cycle} and \ref{row-nilp}. 
\end{proof}

\begin{remark}\rm{In the case that $r_0 \ge \frac{n-2}{n},$ we have $\frac{n-2}{2}+\frac{1}{1-r_0} \ge n-1.$ In that case, the expression for $m(P)$ simplifies to $\min_{\sigma \in S_{n-1}} \left\{ n-1-\frac{\sum_{j=1}^{n-1}j(j+1) r_{\sigma(j)}}{2(1+\sum_{j=1}^{n-1}j r_{\sigma(j)})} \right\}.$ 
	}	
\end{remark}

On its face, Theorem \ref{row-main} leaves us in the position of evaluating  $(n-1)! + 1$ expressions, then choosing the minimum. Our next task is to reduce the number of permutations to be considered in Theorem \ref{row-main}. The following result is useful in that regard.

\begin{proposition}\label{swap}
	Suppose that $1\le j_1 < j_2 \le n-1$ 
	Form the sequence $\hat r_1, \hat r_2, \ldots, \hat r_{n-1}$ from $r_1, r_2, \ldots, r_{n-1} $ by exchanging  $r_{j_1}$ and $r_{j_2}$. Then $\frac{\sum_{j=0}^{n-1}j(j+1)\hat r_j}{\sum_{j=0}^{n-1}(j+1)\hat r_j} - \frac{\sum_{j=0}^{n-1}j(j+1) r_j}{\sum_{j=0}^{n-1}(j+1) r_j}.$ is positive, zero, or negative according as $$(r_{j_1}-r_{j_2})\left(j_1+j_2+1- \frac{\sum_{j=0}^{n-1}j(j+1)r_j }{\sum_{j=0}^{n-1}(j+1)r_j } \right)$$ is positive, zero, or negative.  
\end{proposition}
\begin{proof}
	Observe that $$\frac{\sum_{j=0}^{n-1}j(j+1)\hat r_j}{\sum_{j=0}^{n-1}(j+1)\hat r_j}  = \frac{\sum_{j=0}^{n-1}j(j+1) r_j    +(r_{j_1}-r_{j_2})(j_2-j_1)(j_1+j_2+1)    }{\sum_{j=0}^{n-1}(j+1) r_j +(r_{j_1}-r_{j_2})(j_2-j_1)}.$$ A computation reveals that  
	\begin{flalign*}
&	\frac{\sum_{j=0}^{n-1}j(j+1)\hat r_j}{\sum_{j=0}^{n-1}(j+1)\hat r_j} - \frac{\sum_{j=0}^{n-1}j(j+1) r_j}{\sum_{j=0}^{n-1}(j+1) r_j} =\\ 
&	\frac{(r_{j_1}-r_{j_2})(j_2-j_1) \sum_{j=0}^{n-1}(j+1)r_j(j_1+j_2+1-j)}{\sum_{j=0}^{n-1}(j+1)\hat r_j \sum_{j=0}^{n-1}(j+1) r_j}.
	\end{flalign*}	
	
	%has the same sign as $(r_{j_1}-r_{j_2})(j_2-j_1) \sum_{j=0}^{n-1}(j+1)r_j(j_1+j_2+1-j).$ Since $j_1+j_2+1-j \ge 0$ for each $j=0, \ldots, n-1,$ 
	The conclusion follows. 
\end{proof}

We always have $\frac{\sum_{j=1}^{n-1}j(j+1) r_j}{1+\sum_{j=1}^{n-1}j r_j} \le n-1,$ with equality  if and only if $r_{n-1}=1, r_j=0, j=0, \ldots, n-1.$ That case is excluded from consideration in the following result.

\begin{proposition}\label{prop:order} 
	Suppose that $r_0, r_1, \ldots, r_{n-1} \ge 0, \sum_{j=0}^n r_j=1$ and that $\gamma \equiv  \frac{\sum_{j=1}^{n-1}j(j+1) r_j}{1+\sum_{j=1}^{n-1}j r_j}$ maximizes $\frac{\sum_{j=1}^{n-1}j(j+1) r_{\sigma(j)}}{1+\sum_{j=1}^{n-1}j r_{\sigma(j)}}$ over all permutations $\sigma$ of $\{1, \ldots, n-1\}.$ Let $\rho_j, j=1, \ldots, n-1$ denote the sequence $r_1, \ldots, r_{n-1}$ written in nondecreasing order. Assume that $\gamma < n-1.$  \\
	a) If $\gamma < 4,$ then $r_j=\rho_j, j=1, \ldots, n-1.$\\
	b) If $\gamma \notin \naturals, \gamma >4,$ let $k=\lfloor \gamma \rfloor.$ Then \begin{eqnarray}\label{canon} \nonumber 
		&r_j& =\rho_{k-2j}, j=1, \ldots, \left \lfloor \frac{k-1}{2}\right \rfloor; 
		r_j=\rho_{2j-k+1}, j=\left \lceil \frac{k}{2} \right \rceil, \ldots, k-2;\\ 
		&r_j&=\rho_j, j=k-1, \ldots, n-1.
	\end{eqnarray}\\
	c) If $\gamma = k \in \naturals$ for some $k=4, \ldots, n-2,$ set $\ell = \lfloor \frac{k}{2} \rfloor.$ The ordering given by \eqref{canon}  maximizes $\frac{\sum_{j=1}^{n-1}j(j+1) r_{\sigma(j)}}{1+\sum_{j=1}^{n-1}j r_{\sigma(j)}}$.  
	The other maximizing orderings can be obtained from  \eqref{canon} by choosing a subset $\{ j_1, \ldots, j_m \} \subseteq \{1, \ldots, \ell-1\}$ and exchanging  $r_{j_i}$ and $r_{k-1-j_i}, i=1, \ldots, m$.  
\end{proposition}
\begin{proof}
	The key observation is that from Proposition \ref{swap}, it must be the case  that for indices $1\le j_1 < j_2 \le n-1, (r_{j_1}-r_{j_2})(j_1+j_2+1-\gamma) \le 0.$ \\ 
	a) If $\gamma <4$, then  $j_1+j_2+1-\gamma > 0,$ and the conclusion follows. \\
	b) Since $\gamma \notin \naturals, k<\gamma <k+1.$ From Proposition \ref{swap} we find that $r_j \ge r_{k-j-1}$ for $j=1, \ldots, \lfloor\frac{k-2}{2}\rfloor, r_j \le r_{k-j} $ for $j=1, \ldots, \lfloor\frac{k-1}{2}\rfloor,$ and $r_j \le r_{j+1}$ for $j \ge \lfloor\frac{k-1}{2}\rfloor.$ 
	
	Suppose that $k$ is odd, with $k=2\ell+1$. Then $  r_{2\ell-j+1}\ge  r_j\ge r_{2\ell-j}, j=1, \ldots, \ell-1$ and $r_j \le r_{j+1}, j\ge \ell.$ It now follows that $r_j=\rho_{2\ell-2j+1}, j=1, \ldots, \ell, r_j=\rho_{2j-2\ell}, j=\ell+1, \ldots, 2\ell -1,$ and $r_j=\rho_j, j \ge 2\ell.$ 
	%
	%Similarly, if $k$ is even with $k=2\ell,$ we have $r_{2\ell-j} \ge r_j \ge r_{2\ell-j-1}, j=1, \ldots , \ell-1$ and $r_j\le r_{j+1}, j \ge \ell-1.$ We thus find that $r_j=\rho_{2\ell-2j}, j=1, \ldots, \ell-1, r_j=\rho_{2j-2\ell+1}, j=\ell, \ldots, 2\ell-2,$ and $r_j=\rho_j, j \ge 2\ell-1.$ 
	A similar argument applies if $k$ is even, and we may 
	 then express the odd and even cases in unified form as \eqref{canon}. \\
%	\begin{eqnarray}\label{canonx} \nonumber 
%		r_j &=&\rho_{k-2j}, j=1, \ldots, \left \lfloor \frac{k-1}{2}\right %\rfloor; \\\nonumber 
%		r_j&=&\rho_{2j-k+1}, j=\left \lceil \frac{k}{2} \right \rceil, \ldots, %k-2;\\ 
%		r_j&=&\rho_j, j=k-1, \ldots, n-1.
%	\end{eqnarray}
c)  From Proposition \ref{swap} we find that $r_j \ge r_{k-j-2}, j=1, \ldots, \lfloor \frac{k-3}{2}\rfloor, r_j \le r_{k-j}, j=1, \ldots, \lfloor \frac{k-1}{2}\rfloor$ and $r_j \le r_{j+1}, j \ge  \frac{k-1}{2}\rfloor.$ Note also from Proposition \ref{swap} that for $ j=1, \ldots, \lfloor \frac{k-3}{2}\rfloor$ we have $r_{k-j}\ge r_{k-j-1}\ge r_{k-j-2}.$ 

If $k=2\ell+1$ is odd, we then have $r_j \ge r_{2\ell-j-1}, j=1, \ldots, \ell-1, r_j\le  r_{2\ell-j+1}, j=1, \ldots, \ell, r_j\le r_{j+1}, j \ge \ell.$ 
We also have $r_{2\ell-j+1}\ge r_{2\ell -j}\ge r_{2 \ell -j-1}, j=1, \ldots, \ell-1.$ In particular, $r_{2 \ell -j+1}\ge r_j, r_{2 \ell -j}\ge r_{2 \ell -j-1}, j=1, \ldots, \ell-1$ and 
$r_j\le r_{j+1}, j \ge \ell.$ If it happens to be the case that $ r_j \ge r_{2 \ell -j}$ for each $j=1, \ldots, \ell-1,$ then we recover the ordering  \eqref{canon}. On the other hand, if there are indices $j_1, \ldots, j_m \in \{1, \ldots, \ell-1\}$ for which  $ r_{j_i} < r_{2 \ell -j_i}, i=1, \ldots, m$ we may perform a sequence of $m$ switches that exchange the entries  $r_{j_i}$ and $r_{2\ell -j_i}, i=1, \ldots, m$. According to Proposition \ref{swap}, those exchanges leave the value of the objective function unchanged. We deduce that any maximizing ordering can be constructed by starting with the ordering in \eqref{canon}, then possibly exchanging  $r_{j_i}$ and $r_{2\ell -j_i}, i=1, \ldots, m$ for indices 
$j_1, \ldots, j_m \in \{1, \ldots, \ell-1\}.$

%SHOULD DO AN EXAMPLE OF THIS SWITCHING

If $k=2\ell$ is even, then as above we find that $r_j\ge r_{2\ell-j-2}, j=1, \ldots, \ell-2, r_j \le r_{2\ell-j}, j=1, \ldots, \ell-1, r_j \le r_{j+1}, j \ge \ell-1,$ and $r_{2\ell-j} \ge r_{2\ell-j-1} \ge  r_{2\ell-j-2}, j=1, \ldots, \ell-2.$ If it happens to be the case that $r_j \ge r_{2\ell-j-1}, j=1, \ldots, \ell-1,$ then we generate the ordering \eqref{canon}.  On the other hand, if there are indices $j_1, \ldots, j_m \in \{1, \ldots, \ell-1\}$ for which  $ r_{j_i} < r_{2 \ell -j_i-1}, i=1, \ldots, m$ we may perform a sequence of $m$ switches that exchange  $r_{j_i}$ and $r_{2\ell -j_i-1}, i=1, \ldots, m$. From Proposition \ref{swap}, those exchanges leave the value of the objective function unchanged. We deduce that any maximizing ordering can be constructed by starting with  \eqref{canon}, then possibly exchanging  $r_{j_i}$ and $r_{2\ell -j_i-1}, i=1, \ldots, m$ for indices 
$j_1, \ldots, j_m \in \{1, \ldots, \ell-1\}.$
\end{proof}

\begin{remark}
{\rm{We emphasize that the ordering in \eqref{canon} depends on the specific value of $k$. 
}}	
\end{remark}

\begin{corollary}\label{cor:min} 
	Suppose that $P$ is an $n \times n$ partial  stochastic matrix where row $n$ has been specified as $p_{n,j}=r_{n-j}, j=1, \ldots, n,$ and all remaining entries are unspecified. Suppose further that $r_j<1, j=0, \ldots, n-1.$ Let $\rho_j, j=1, \ldots, n-1$ denote the sequence $r_1, \ldots, r_{n-1}$ written in nondecreasing order. For each $k=4, \ldots, n-2$, let $r_j^{(k)}$ denote the ordering of \eqref{canon} associated with  $k$. Then 
	\begin{flalign*}
		&	m(P) = 
			\min \left\{ \frac{n-2}{2}+\frac{1}{1-r_0},  n-1-\frac{\sum_{j=1}^{n-1}j(j+1) \rho_j}{2(1+\sum_{j=1}^{n-1}j \rho_j)} , \right.\\ 	&	\left. n-1-\frac{\sum_{j=1}^{n-1}j(j+1) r_j^{(4)}}{2(1+\sum_{j=1}^{n-1}j r_j^{(4)})}, 
	 \ldots, 
		n-1-\frac{\sum_{j=1}^{n-1}j(j+1) r_j^{(n-2)}}{2(1+\sum_{j=1}^{n-1}j r_j^{(n-2)})} \right\}.
	\end{flalign*}
\end{corollary}

\begin{corollary}\label{incr} 
	Suppose that $3\le n \le 5.$  Then $$m(P)=\min \left\{\frac{n-2}{2} +\frac{1}{1-r_0},n-1-  \frac{\sum_{j=0}^{n-1}j(j+1) \rho_j}{2\sum_{j=0}^{n-1}(j+1) \rho_{j}} \right \}.$$ 
\end{corollary}

\begin{remark}{\rm{In this remark, we show how, for each $k=4, \ldots, n-2$ there is a sequence $r_0, r_1, \ldots, r_{n-1}$ such that the optimal ordering corresponds to that given by \eqref{canon} for the parameter $k$. Fix such a $k$ and suppose that $\gamma \in [k,k+1).$ 
			
Suppose that $\epsilon \in (0,\frac{1}{2})$ and select a vector $c \in \reals^{n-2}$ with $c_j\ge 0, j=1, \ldots, n-2$ and $\sum_{j=1}^{n-2} c_j=1.$ Now set $\hat r_j=\epsilon c_j, j=1, \ldots, n-2$ and $\hat r_{n-1}=1-\epsilon.$ Fix a $k$ with $4\le k \le n-2$ and assume that $\hat r_1, \ldots, \hat r_{n-1}$ has the  ordering given by \eqref{canon} corresponding to $k.$ 
Fix $\gamma \in [k, k+1),$ define $r_0$ via $$1-r_0 = \frac{\gamma}{(1-\epsilon)(n-1)(n-\gamma) + \epsilon(\sum_{j=1}^{n-2}
jc_j(j+1-\gamma)}$$ and set $r_j=(1-r_0)\hat r_j, j=1, \ldots , n-1.$ (We note that $r_0 \in (0,1)$ for all sufficiently small $\epsilon>0.$) For notational convenience we suppress the dependence on $\epsilon$ and $\gamma.$ It is straightforward to determine that $\frac{\sum_{j=0}^{n-1} j(j+1)r_j}{\sum_{j=0}^{n-1} (j+1)r_j} = \gamma.$ Observe also that for all sufficiently small $\epsilon>0$ and any permutation $\sigma $ of $\{1, \ldots, n-1\},$ $\frac{\sum_{j=1}^{n-1} j(j+1)r_{\sigma(j)}}{r_0+\sum_{j=1}^{n-1} (j+1)r_{\sigma(j)}} = \gamma + O(\epsilon).$ In particular, for any such $\epsilon$ and $\sigma, \frac{\sum_{j=1}^{n-1} j(j+1)r_{\sigma(j)}}{r_0+\sum_{j=1}^{n-1} (j+1)r_{\sigma(j)}}< k+1. $ We deduce that $\max \left \{ \frac{\sum_{j=1}^{n-1} j(j+1)r_{\sigma(j)}}{r_0+\sum_{j=1}^{n-1} (j+1)r_{\sigma(j)}} \Bigg|\sigma \in S_{n-1}  \right \} \in [\gamma, k+1).$  If $k<\gamma,$ then Proposition \ref{prop:order} b) applies, and hence the unique ordering of $r_1, \ldots, r_{n-1}$ that yields the maximum is in fact given by \eqref{canon}. 
	
Suppose now that $\gamma=k.$ If it were the case that $$\max \left \{ \frac{\sum_{j=1}^{n-1} j(j+1)r_{\sigma(j)}}{r_0+\sum_{j=1}^{n-1} (j+1)r_{\sigma(j)}} \Bigg|\sigma \in S_{n-1}  \right \} > \gamma,$$ then  Proposition \ref{prop:order} b) would necessitate that the  ordering \eqref{canon} yields the maximum value, a contradiction (since that ordering yields the value $\gamma$ ). Consequently, we find that $\max \left \{ \frac{\sum_{j=1}^{n-1} j(j+1)r_{\sigma(j)}}{r_0+\sum_{j=1}^{n-1} (j+1)r_{\sigma(j)}} \Bigg|\sigma \in S_{n-1}  \right \} = k.$

Finally, we note that for this sequence, and sufficiently small $\epsilon>0$, 
we have $\frac{n}{2}+ \frac{1}{1-r_0}= \frac{n}{2} +\frac{(n-1)(n-\gamma)}{\gamma} + O(\epsilon),$ while the ordering given by \eqref{canon} corresponds to $n-1-\frac{\gamma}{2}.$ It now follows that $\frac{n}{2}+ \frac{1}{1-r_0}$ is strictly greater than the minimum value of Kemeny's constant. 
}}	
\end{remark}

%condition on $r_0$ to ensure that cycle version is the winner. 

%\begin{corollary} ***NEEDS REPAIR***
%	There is   a permutation $\sigma$  of $\{1, \ldots, n-1\}$ 
%	that maximizes 
%	suppose that 
%	$ \frac{\sum_{j=0}^{n-1}j(j+1) r_{\tau(j)}}{\sum_{j=0}^{n-1}(j+1) r_{\tau(j)}}$ over all  permutations $\tau$ of $\{1, \ldots, n-1\}$ such that  $r_{\sigma(j)}$ is nondecreasing in $j$  for 
%	$ \lceil \frac{n-3}{2}\rceil \le j \le n-1$. 
%\end{corollary}
%\begin{proof}
%	Fix a permutation $\tau $ of  $\{1, \ldots, n-1\}$ and 
%	set $\rho_j=r_{\tau(j)}, j=1, \ldots, n-1.$ If it is not the case that $\rho_j$ is nondecreasing for in $j$  for 
%	$ \lceil \frac{n-3}{2}\rceil \le j \le n-1$ then for some $\lceil \frac{n-3}{2}\rceil \le j_1<j_2 \le n-1$ we have $\rho_{j_1}> \rho_{j_2}.$ Further, we have $j_1+j_2\le 2j_1+1\ge n-2,$  so applying Proposition \ref{swap}, we can exchange $\rho_{j_1}$ and   $\rho_{j_2}$ to generate a new permutation $\hat \tau$ so that $\frac{\sum_{j=0}^{n-1}j(j+1) r_{\hat \tau(j)}}{\sum_{j=0}^{n-1}(j+1) r_{\hat \tau(j)}} \ge \frac{\sum_{j=0}^{n-1}j(j+1) r_{\tau(j)}}{\sum_{j=0}^{n-1}(j+1) r_{\tau(j)}}.$ Iterating the procedure as needed, we reach the desired conclusion. 
%\end{proof}

We close the paper with a brief numerical illustration. 

\begin{example}\rm{
	Suppose that $n=6$ and $r_0=0.$ Consider the values $\frac{1}{3}, \frac{1}{3}, \frac{1}{6}, \frac{1}{6}, 0$ as the remaining specified entries in last row of the partial stochastic matrix $P$. From Corollary \ref{cor:min}, it suffices to consider $r_1, r_2, r_3, r_4$ in one of the following two orders: $0, \frac{1}{6}, \frac{1}{6}, \frac{1}{3}, \frac{1}{3}$; $ \frac{1}{6}, 0, \frac{1}{6}, \frac{1}{3}, \frac{1}{3}.$ The first ordering yields $\frac{86}{29} \approx 2.9655$ for Kemeny's constant, while the second ordering yields $\frac{83}{28} \approx 2.9643$ for Kemeny's constant. Observing that in this case $\frac{n-2}{2}+\frac{1}{1-r_0}=3,$ we find that $m(P)=\frac{83}{28}.$ 
	%In particular, the minimum value for Kemeny's constant corresponds to an %ordering of $r_1, \ldots, r_5$ that is not monotonic. 
}	
\end{example}

%\section{Specified column} \label{sec:col}

%column with positive entries also has a max for kemeny completion. 

\bigskip
\noindent
{\emph{Acknowledgement}}: 
The author's research is supported in part by the Natural Sciences and Engineering Research Council
of Canada under grant number RGPIN–2019–05408. 

%\begin{Backmatter} 

%\printaddress

%\end{Backmatter}

\begin{thebibliography}{99}

%\bibitem{B} R. Bapat,  \emph{Graphs and Matrices}, Springer, London, %2010. 

\bibitem{CKS} E. Crisostomi, S. Kirkland and R. Shorten,  A Google-like model of road network dynamics and its application to regulation and control, \textit{Internat. J. Control} 84 (2011), 633--651.\\  
\url{https://doi.org/10.1080/00207179.2011.568005} 



\bibitem{hla} 
L. De Alba, L. Hogben and A. Wangsness Wehe, Matrix completion problems, in \textit{Handbook of Linear Algebra, 2nd Edition}, 
 Leslie Hogben Ed., CRC Press, Boca Raton, 2014, pp.  49-1--49-30.\\
\url{ https://doi.org/10.1201/b16113-56 }
 
 % Second edition
%Discrete Math. Appl. (Boca Raton)
 %xxx+1874 pp.





%\bibitem{DHFS} P. Deuflhard, W. Huisinga,  A. Fischer and C. %Sch\"{u}tte, Identification of almost invariant aggregates in 
%reversible nearly uncoupled Markov chains, \textit{Linear Algebra and %its Applications} 315 (2000) 39--59. 


\bibitem{hj} R. Horn and C. Johnson, \textit{Matrix Analysis, 2nd Edition}, Cambridge University Press, Cambridge, 2013.  \\
\url{https://doi.org/10.1017/CBO9781139020411 } 

\bibitem{hunter} J. Hunter, Mixing times with applications to perturbed
Markov chains, \textit{Linear Algebra  Appl.} 417 (2006) 108--123. \\
\url{https://doi.org/10.1016/j.laa.2006.02.008}  

\bibitem{Jo}
C. Johnson, 
Matrix completion problems: a survey, in \textit{Matrix Theory and Applications, Proceedings of Symposia in Applied Mathematics} 40 (1990),  pp. 171--198.\\
\url{https://doi.org/10.1090/psapm/040/1059486 } 

%Proc. Sympos. Appl. Math., 40
%AMS Short Course Lecture Notes
%American Mathematical Society, Providence, RI, 1990


	\bibitem{KS} 
	J. Kemeny and J. Snell,  
	\textit{Finite Markov Chains},  
	Springer-Verlag, New York, 1976.
%\bibitem{ks} Kemeny and Snell, Finite Markov Chains
%\bibitem{kz} Kirkland and Zeng
%\bibitem{ll} Levene and Loizeau
\bibitem{kfast} S. Kirkland, Fastest expected time to mixing for a Markov chain on a directed graph, \textit{Linear Algebra Appl.}  433 (2010), 1988--1996.\\
\url{https://doi.org/10.1016/j.laa.2010.07.016 } 

\bibitem{Kstat} S. Kirkland, On the Kemeny constant and stationary distribution vector for a Markov chain, \textit{
Electron. J. Linear Algebra} 27 (2014), 354--372. \\
\url{https://doi.org/10.13001/1081-3810.1624 }

\bibitem{kczech} 
S. Kirkland, Random walk centrality and a partition of Kemeny's constant, \textit{Czechoslovak Math. J.} 66(141) (2016), 757--775. \\
\url{https://doi.org/10.1007/s10587-016-0291-9 } 

	
\bibitem{KN}
S. Kirkland and M. Neumann, 
\textit{Group Inverses of M-Matrices and their Applications}, 
CRC Press, Boca Raton, 2013.  \\
\url{https://doi.org/10.1201/b13054 }

%	\bibitem{kz}
%	S. Kirkland and Z. Zeng, 
%	Kemeny's constant and an analogue of Braess' paradox for trees, 
%	\textit{Electronic Journal of Linear Algebra} 31 (2016), 444--464.  

%\bibitem{langmey} A. Langville and C. Meyer, \textit{Google's PageRank %and Beyond: the Science of Search Engine Rankings},   Princeton %University Press,  Princeton,  2006. 

	\bibitem{ll} 
	M. Levene and G. Loizou, Kemeny's constant and the random surfer,  \textit{Amer. Math. Monthly} 109 (2002),  741--745. \\
\url{https://doi.org/10.2307/3072398 }

\bibitem{sen} E. Seneta, \textit{Non-negative Matrices and Markov Chains, 2nd Edition}, Springer, New York, 1981. \\
\url{https://doi.org/10.1007/0-387-32792-4 }
 \end{thebibliography}
 \end{document}